\documentclass[12pt]{article}

\usepackage[english]{babel}
\usepackage{graphicx}
\usepackage{latexsym}
\usepackage{amsfonts}
\usepackage{dsfont}
\usepackage{amscd}
\usepackage{amssymb}
\usepackage{amsmath}
\usepackage{hyperref}
\usepackage{mathrsfs}
\usepackage{amsthm}
\usepackage{tikz}
\usepackage{enumitem}

\input xy
\xyoption{all}

\newlength{\bibitemsep}
\newlength{\bibparskip}
\let\oldthebibliography\thebibliography
\renewcommand\thebibliography[1]{%
  \oldthebibliography{#1}%
  \setlength{\parskip}{\bibitemsep}%
  \setlength{\itemsep}{\bibparskip}%
}

\newcommand{\RR}{\mathbb{R}}
\newcommand{\NN}{\mathbb{N}}
\newcommand{\CC}{\mathbb{C}}
\newcommand{\ZZ}{\mathbb{Z}}
\newcommand{\Op}{\Omega_{\text{punc}}}
\newcommand{\Rp}{\mathcal{R}_{\text{punc}}}

\newcommand{\CB}{C^*_B}
\newcommand{\M}{\mathbb{M}}
\newcommand{\I}{\mathcal{I}}
\newcommand{\p}{\mathcal{P}}

\newcommand{\J}{\mathcal{J}}

\newcommand{\g}{\mathcal{G}}
\newcommand{\Ef}{\widehat E_0}
\newcommand{\Eu}{\widehat E_\infty}
\newcommand{\Et}{\widehat E_{\text{tight}}}
\newcommand{\Ct}{C^*_{\text{tight}}}
\newcommand{\Cr}{C^*_{\text{red}}}
\newcommand{\gt}{\mathcal{G}_{\text{tight}}}

\newcommand{\fs}{\subset_{\text{fin}}}

\setlist[enumerate]{itemsep=-1pt}
\setlist[itemize]{itemsep=-1pt}
\theoremstyle{definition}
\newtheorem{theo}{Theorem}[section]
\newtheorem{rmk}[theo]{Remark}

\newtheorem{ex}[theo]{Example}
\newtheorem{lem}[theo]{Lemma}
\newtheorem{defn}[theo]{Definition}
\newtheorem{cor}[theo]{Corollary}

\newtheorem{prop}[theo]{Proposition}
\date{}
\begin{document}
\title{C*-algebras of Boolean inverse monoids -- traces and invariant means}
\author{Charles Starling\thanks{Supported by the NSERC grants of Beno\^it Collins, Thierry Giordano, and Vladimir Pestov. \texttt{cstar050@uottawa.ca}.}}
\maketitle

\begin{abstract}
To a Boolean inverse monoid $S$ we associate a universal C*-algebra $C^*_B(S)$ and show that it is equal to Exel's tight C*-algebra of $S$. We then show that any invariant mean on $S$ (in the sense of Kudryavtseva, Lawson, Lenz and Resende) gives rise to a trace on $C^*_B(S)$, and vice-versa, under a condition on $S$ equivalent to the underlying groupoid being Hausdorff. Under certain mild conditions, the space of traces of $C^*_B(S)$ is shown to be isomorphic to the space of invariant means of $S$. We then use many known results about traces of C*-algebras to draw conclusions about invariant means on Boolean inverse monoids; in particular we quote a result of Blackadar to show that any metrizable Choquet simplex arises as the space of invariant means for some AF inverse monoid $S$.
\end{abstract}

\section{Introduction}
This article is the continuation of our study of the relationship between inverse semigroups and C*-algebras. An {\em inverse semigroup} is a semigroup $S$ for which every element $s\in S$ has a unique ``inverse'' $s^*$ in the sense that 
\[
ss^*s = s \text{ and }s^*ss^* = s^*.
\]
An important subsemigroup of any inverse semigroup is its set of idempotents $E(S) = \{e\in S\mid e^2 = e\} = \{s^*s\mid s\in S\}$. Any set of partial isometries closed under product and involution inside a C*-algebra is an inverse semigroup, and its set of idempotents forms a commuting set of projections. Many C*-algebras $A$ have been profitably studied in the following way:
\begin{enumerate}
\item identify a generating inverse semigroup $S$,
\item write down an abstract characterization of $S$,
\item show that $A$ is universal for some class of representations of $S$.
\end{enumerate}
We say ``some class'' above because typically considering all representations (as in the construction of Paterson \cite{Pa99}) gives us a larger C*-algebra than we started with. For example, consider the multiplicative semigroup inside the Cuntz algebra $\mathcal{O}_2$ generated by the two canonical generators $s_0$ and $s_1$; in semigroup literature this is usually denoted $P_2$ and called the {\em polycyclic monoid} of order 2. The C*-algebra which is universal for all representations of $P_2$ is $\mathcal{T}_2$, the Toeplitz extension of $\mathcal{O}_2$. In an effort to arrive back at the original C*-algebra in cases such as this, Exel defined the notion of {\em tight} representations \cite{Ex08}, and showed that the universal C*-algebras for tight representations of $P_2$ is $\mathcal{O}_2$. See \cite{ShiftISG}, \cite{StLCM}, \cite{EP16}, \cite{EP14a}, \cite{EGS12}, \cite{COP15} for other examples of this approach.

Another approach to this issue is to instead alter the inverse semigroup $S$. An inverse semigroup carries with it a natural order structure, and when an inverse semigroup $S$ is represented in a C*-algebra $A$, two elements $s, t\in S$, which did not have a lowest upper bound in $S$, may have one inside $A$. So, from $P_2$, Lawson and Scott \cite[Proposition 3.32]{LS14} constructed a new inverse semigroup $C_2$, called the {\em Cuntz inverse monoid}, by adding to $P_2$ all possible joins of {\em compatible} elements ($s, t$ are compatible if $s^*t, st^*\in E(S)$). 

The Cuntz inverse monoid is an example of a {\em Boolean inverse monoid}, and the goal of this paper is to define universal C*-algebras for such monoids and study them. A Boolean inverse monoid is an inverse semigroup which contains joins of all finite compatible sets of elements and whose idempotent set is a Boolean algebra. To properly represent a Boolean inverse monoid $S$, one reasons, one should insist that the join of two compatible $s, t\in S$ be sent to the join of the images of $s$ and $t$. We prove in Proposition \ref{BIMRtight} that such a representation is necessarily a tight representation, and so we obtain that the universal C*-algebra of a Boolean inverse monoid (which we denote $\CB(S)$) is exactly its tight C*-algebra, Theorem \ref{BIMtight}. This is the starting point of our study, as the universal tight C*-algebra can be realized as the C*-algebra of an ample groupoid. 

The main inspiration of this paper is \cite{KLLR15} which defines and studies {\em invariant means} on Boolean inverse monoids. An invariant mean is a function $\mu: E(S)\to [0,\infty)$ such that $\mu(e\vee f) = \mu(e)+\mu(f)$ when $e$ and $f$ are orthogonal, and such that $\mu(ss^*) = \mu(s^*s)$ for all $s\in S$. If one thinks of the idempotents as clopen sets in the Stone space of the Boolean algebra $E(S)$, such a function has the flavour of an invariant measure or a trace. We make this precise in Section 4: as long as $S$ satisfies a condition which guarantees that the induced groupoid is Hausdorff (which we call condition \eqref{condH}), every invariant mean on $S$ gives rise to a trace on $\CB(S)$ (Proposition \ref{taumu}) and every trace on $\CB(S)$ gives rise to an invariant mean on $S$ (Proposition \ref{mutau}). This becomes a one-to-one correspondence if we assume that the associated groupoid $\gt(S)$ is principal and amenable (Theorem \ref{maintheo}). We also prove that, whether $\gt(S)$ is principal and amenable or not, there is an affine isomorphism between the space of invariant means on $S$ and the space of $\gt(S)$-invariant measures on its unit space (Proposition \ref{etanu}).

In the final section, we apply our results to examples of interest. We study the {\em AF inverse monoids} in detail -- these are Boolean inverse monoids arising from Bratteli diagrams in much the same way as AF C*-algebras. As it should be, given a Bratteli diagram, the C*-algebra of its Boolean inverse monoid is isomorphic to the AF algebra it determines (Theorem \ref{AFBIM}). From this we can conclude, using the results of Section 4 and the seminal result of Blackadar \cite{Bl80}, that any Choquet simplex arises as the space of invariant means for some Boolean inverse monoid. We go on to consider two examples where there is typically only one invariant mean, those being self-similar groups and aperiodic tilings.

\section{Preliminaries and notation}
We will use the following general notation. If $X$ is a set and $U\subset X$, let Id$_U$ denote the map from $U$ to $U$ which fixes every point, and let $1_{U}$ denote the characteristic function on $U$, i.e. $1_U: X\to \CC$ defined by $1_U(x) = 1$ if $x\in U$ and $1_U(x) = 0$ if $x\notin U$. If $F$ is a finite subset of $X$, we write $F\fs X$.
\subsection{Inverse semigroups}\label{ISGsubsection}

An {\em inverse semigroup} is a semigroup $S$ such that for all $s\in S$, there is a unique element $s^*\in S$ such that
\[
ss^*s = s, \hspace{1cm}s^*ss^* = s^*.
\]
The element $s^*$ is called the {\em inverse} of $s$. All inverse semigroups in this paper are assumed to be discrete and countable. For $s,t\in S$, one has $(s^*)^* = s$ and $(st)^* = t^*s^*$. Although not implied by the definition, we will always assume that inverse semigroups have a 0 element, that is, an element such that
\[
0s = s0 = 0\text{ for all }s\in S.
\]
An inverse semigroup with identity is called an {\em inverse monoid}. Even though we call $s^*$ the inverse of $s$, we need not have $ss^* = 1$, although it is always true that $(ss^*)^2 = ss^*ss^* = ss^*$, i.e. $ss^*$ (and $s^*s$ for that matter) is an {\em idempotent}. We denote the set of all idempotents in $S$ by
\[
E(S) = \{e\in S\mid e^2 = e\}.
\]
It is a nontrivial fact that if $S$ is an inverse semigroup, then $E(S)$ is closed under multiplication and commutative. It is also clear that if $e\in E(S)$, then $e^* = e$.

Let $X$ be a set, and let
\[
\I(X) = \{f:U\to V\mid U, V\subset X, f\text{ bijective}\}.
\]
Then $\I(X)$ is an inverse monoid with the operation of composition on the largest possible domain, and inverse given by function inverse; this is called the {\em symmetric inverse monoid on $X$}.  Every idempotent in $\I(X)$ is given by Id$_U$ for some $U\subset X$. The function Id$_X$ is the identity for $\I(X)$, and the empty function is the 0 element for $\I(X)$. The fundamental {\em Wagner-Preston theorem} states that every inverse semigroup is embeddable in $\I(X)$ for some set $X$ -- one can think of this as analogous to the Cayley theorem for groups.

Every inverse semigroup carries a natural order structure: for $s,t\in S$ we say $s\leqslant t$ if and only if $ts^*s = s$, which is also equivalent to $ss^*t = s$. For elements $e, f\in E(S)$, we have $e\leqslant f$ if and only if $ef = e$. As usual, for $s,t\in S$, the {\em join} (or least upper bound) of $s$ and $t$ will be denoted $s\vee t$ (if it exists), and the {\em meet} (or greatest lower bound) of $s$ and $t$ will be denoted $s\wedge t$ (if it exists). For $A\subset S$, we let $A^{\uparrow} = \{t\in S\mid s\leqslant t \text{ for some }s\in A\}$ and $A^{\downarrow} = \{t\in S\mid t\leqslant s\text{ for some }s\in A\}$.

If $s,t\in S$, then we say $s$ and $t$ are {\em compatible} if $s^*t, st^*\in E(S)$, and a set $F\subset S$ is called compatible if all pairs of elements of $F$ are compatible. 

\begin{defn}
An inverse semigroup $S$ is called {\em distributive} if whenever we have a compatible set $F\fs S$, then $\bigvee_{s\in F}s$ exists in $S$, and for all $t\in S$ we have
\[
t\left(\bigvee_{s\in F}s\right) = \bigvee_{s\in F}ts\hspace{1cm}\text{and}\hspace{1cm} \left(\bigvee_{s\in F}s\right)t = \bigvee_{s\in F}st.
\]
\end{defn}

In the natural partial order, the idempotents form a {\em meet semilattice}, which is to say that any two elements $e, f\in E(S)$ have a meet, namely $ef$. If $C\subset X\subset  E(S)$, we say that $C$ is a {\em cover} of $X$ if for all $x\in X$ there exists $c\in C$ such that $cx \neq 0$.

In a distributive inverse semigroup each pair of idempotents has a join in addition to the meet mentioned above, but in general $E(S)$ will not have relative complements and so in general will not be a Boolean algebra. The case where $E(S)$ is a Boolean algebra is the subject of the present paper.

\begin{defn}
A {\em Boolean inverse monoid} is a distributive inverse monoid $S$ with the property that $E(S)$ is a Boolean algebra, that is, for every $e\in E(S)$ there exists $e^\perp\in E(S)$ such that $ee^\perp = 0$, $e\vee e^\perp = 1$, and the operations $\vee, \wedge, \perp$ satisfy the laws of a Boolean algebra \cite[Chapter 2]{GH09}.
\end{defn} 

\begin{ex}
Perhaps the best way to think about the order structure and related concepts above is by describing them on $\I(X)$, which turns out to be a Boolean inverse monoid. Firstly, for $g,h\in \I(X)$, $g\leqslant h$ if and only if $h$ extends $g$ as a function. In $\I(X)$, two functions $f$ and $g$ are compatible if they agree on the intersection of their domains and their inverses agree on the intersection of their ranges. In such a situation, one can form the join $f\vee g$ which is the union of the two functions; this will again be an element of $\I(X)$. Composing $h\in \I(X)$ with $f\vee g$ will be the same as $hf\vee hg$. Finally, $E(\I(X)) = \{\text{Id}_U\mid U\subset X\}$ is a Boolean algebra (isomorphic to the Boolean algebra of all subsets of $X$) with Id$_U^\perp =$ Id$_{U^c}$. 

\end{ex}
\subsection{\'Etale groupoids}
A {\em groupoid} is a small category where every arrow is invertible. If $\g$ is a groupoid, the set of elements $\gamma\gamma^{-1}$ is denoted $\g^{(0)}$ and is called the set of {\em units} of $\g$. The maps $r: \g \to \g^{(0)}$ and $d:\g \to \g^{(0)}$ defined by $r(\gamma) = \gamma\gamma^{-1}$ and $d(\gamma) = \gamma^{-1}\gamma$ are called the {\em range} and {\em source} maps, respectively.

The set $\g^{(2)} = \{(\gamma, \eta)\in \g^2\mid r(\eta) = d(\gamma)\}$ is called the set of {\em composable pairs}. A {\em topological groupoid} is a groupoid $\g$ which is a topological space and for which the inverse map from $\g$ to $\g$ and the product from $\g^{(2)}$ to $\g$ are both continuous (where in the latter, the topology on $\g^{(2)}$ is the product topology inherited from $\g^2$). 

We say that a topological groupoid $\g$ is {\em \'etale} if it is locally compact, second countable, $\g^{(0)}$ is Hausdorff, and the maps $r$ and $d$ are both local homeomorphisms. Note that an \'etale groupoid need not be Hausdorff. If $\g$ is \'etale, then $\g^{(0)}$ is open, and $\g$ is Hausdorff if and only if $\g^{(0)}$ is closed (see for example \cite[Proposition 3.10]{EP16}).

For $x\in \g^{(0)}$, let $\g(x) = \{\gamma \in \g\mid r(\gamma) = d(\gamma) = x\}$ -- this is a group, and is called the {\em isotropy group at $x$}. A groupoid $\g$ is said to be {\em principal} if all the isotropy groups are trivial, and a topological groupoid is said to be {\em essentially principal} if the points with trivial isotropy groups are dense in $\g^{(0)}$. A topological groupoid is said to be {\em minimal} if for all $x\in \g^{(0)}$, the set $O_\g(x) = r(d^{-1}(x))$ is dense in $\g^{(0)}$ (the set $O_\g(x)$ is called the {\em orbit} of $x$).

If $\g$ is an \'etale groupoid, an open set $U\subset \g$ is called a {\em bisection} if $\left.r\right|_U$ and $\left.d\right|_U$ are both injective (and hence homeomorphisms). The set of all bisections is denoted $\g^{op}$ and is a distributive inverse semigroup when given the operations of setwise product and inverse. We say that an \'etale groupoid $\g$ is {\em ample} if the set of {\em compact} bisections forms a basis for the topology on $\g$. The set of compact bisections is called the {\em ample semigroup of} $\g$, is denoted $\g^a$, and is also a distributive inverse subsemigroup of $\g^{op}$ \cite[Lemma 3.14]{LL13}. Since $\g$ is second countable, $\g^a$ must be countable \cite[Corollary 4.3]{Ex10}. If $\g^{(0)}$ is compact, then the idempotent set of $\g^a$ is the set of all clopen sets in $\g^{(0)}$, and so $\g^a$ is a Boolean inverse monoid (see also \cite[Proposition 3.7]{Ste10} which shows that when $\g$ is Hausdorff and $\g^{(0)}$ is only locally compact, $\g^a$ is a Boolean inverse {\em semigroup}, i.e. a distributive inverse semigroup whose idempotent semilattice is a generalized Boolean algebra).

To an \'etale groupoid $\g$ one can associate C*-algebras through the theory developed by Renault \cite{R80}. Let $C_c(\g)$ denote the linear space of continuous compactly supported functions on $\g$. Then $C_c(\g)$ becomes a $*$-algebra with product and involution given by
\[
fg(\gamma) = \sum_{\gamma_1\gamma_2 = \gamma} f(\gamma_1)g(\gamma_2), \hspace{1cm} f^*(\gamma) = \overline{f(\gamma^{-1})}.
\]
From this one can produce two C*-algebras $C^*(\g)$ and $C_{\text{red}}^*(\g)$ (called the {\em C*-algebra of $\g$} and the {\em reduced C*-algebra of $\g$}, respectively) by completing $C_c(\g)$ in certain norms, see \cite[Definitions 1.12 and 2.8]{R80}. There is always a surjective $*$-homomorphism $\Lambda: C^*(\g)\to C_{\text{red}}^*(\g)$, and if $\Lambda$ is an isomorphism we say that $\g$ satisfies {\em weak containment}. If $\g$ is {\em amenable} \cite{AR00}, then $\g$ satisfies weak containment.  There is an example of a case where $\Lambda$ is an isomorphism for a nonamenable groupoid \cite{Wi15}, but under some conditions on $\g$ one has that weak containment and amenability are equivalent, see \cite[Theorem B]{AD16}.

Recall that if $B\subset A$ are both C*-algebras, then a surjective linear map $E:A \to B$ is called a {\em conditional expectation} if $E$ is contractive, $E\circ E = E$, and $E(bac) = bE(a)c$ for all $b,c\in B$ and $a\in A$. Let $\g$ be a Hausdorff \'etale groupoid with compact unit space, and consider the map $E: C_c(\g)\to C(\g^{(0)})$ defined by 
\begin{equation}\label{conditionalexpectation}
E(f) = \left.f\right|_{\g^{(0)} }.
\end{equation} 
Then this map extends to a conditional expectation on both $C^*(\g)$ and $\Cr(\g)$, both denoted $E$. On $\Cr(\g)$, $E$ is {\em faithful} in the sense that if $E(a^*a) = 0$, then $a=0$.

Let $\g$ be an ample \'etale groupoid. Both C*-algebras contain $C_c(\g)$, and hence if $U$ is a compact bisection, $1_U$ is an element of both C*-algebras. Hence we have a map $\pi:\g^a\to C^*(\g)$ given by $\pi(U) = 1_U$. This map satisfies $\pi(UV) = \pi(U)\pi(V)$, $\pi(U^{-1})$, and $\pi(0) = 0$, in other words, $\pi$ is a {\em representation} of the inverse semigroup $\g^a$ \cite{Ex10}.

\subsection{The tight groupoid of an inverse semigroup}

Let $S$ be an inverse semigroup. A {\em filter} in $E(S)$ is a nonempty subset $\xi\subset E(S)$ such that 
\begin{enumerate}
\item$0\notin \xi$, 
\item$e, f\in \xi$ implies that $ef\in \xi$, and 
\item $e\in \xi, e\leqslant f$ implies $f\in \xi$. 
\end{enumerate}
The set of filters is denoted $\Ef(S)$, and can be viewed as a subspace of $\{0,1\}^{E(S)}$. For $X,Y\fs E(S)$, let
\[
U(X,Y) = \{\xi\in \Ef(S)\mid X\subset \xi, Y\cap \xi = \emptyset\}.
\]
sets of this form are clopen and generate the topology on $\Ef(S)$ as $X$ and $Y$ vary over all the finite subsets of $E(S)$. With this topology, $\Ef(S)$ is called the {\em spectrum} of $E(S)$.

A filter is called an {\em ultrafilter} if it is not properly contained in any other filter. The set of all ultrafilters is denoted $\Eu(S)$. As a subspace of $\Ef(S)$, $\Eu(S)$ may not be closed. Let $\Et(S)$ denote the closure of $\Eu(S)$ in $\Ef(S)$ -- this is called the {\em tight spectrum} of $E(S)$. Of course, when $E(S)$ is a Boolean algebra, $\Et(S) = \Eu(S)$ by Stone duality \cite[Chapter 34]{GH09}.

An {\em action} of an inverse semigroup $S$ on a locally compact space $X$ is a semigroup homomorphism $\alpha: S\to \I(X)$ such that 
\begin{enumerate}
\item $\alpha_s$ is continuous for all $s\in S$,
\item the domain of $\alpha_s$ is open for each $s\in S$, and
\item the union of the domains of the $\alpha_s$ is equal to $X$.
\end{enumerate}
If $\alpha$ is an action of $S$ on $X$, we write $\alpha: S\curvearrowright X$. The above implies that $\alpha_{s^*} = \alpha_s^{-1}$, and so each $\alpha_s$ is a homeomorphism. For each $e\in E(S)$, the map $\alpha_e$ is the identity on some open subset $D_e^\alpha$, and one easily sees that the domain of $\alpha_s$ is $D_{s^*s}^\alpha$ and the range of $\alpha_s$ is $D_{ss^*}^\alpha$, that is
\[
\alpha_s: D_{s^*s}^\alpha\to D_{ss^*}^\alpha.
\]

There is a natural action $\theta$ of $S$ on $\Et(S)$; this is referred to in \cite{EP16} as the {\em standard action} of $S$. For $e\in E(S)$, let $D_e^\theta = \{\xi\in \Et(S)\mid e\in \xi\} = U(\{e\}, \emptyset)\cap \Et(S)$. For each $s\in S$ and $\xi \in D_{s^*s}^\theta$, define $\theta_s(\xi) = \{ses^*\mid e\in \xi\}^{\uparrow}$ -- this is a well-defined homeomorphism from $D_{s^*s}^\theta$ to $D_{ss^*}^\theta$, for the details, see \cite{Ex08}.

One can associate a groupoid to an action $\alpha:S \curvearrowright X$. Let $S\times_\alpha X = \{(s,x)\in S\times X\mid x\in D_{s^*s}^\alpha\}$, and put an equivalence relation $\sim$ on this set by saying that $(s,x)\sim (t, y)$ if and only if $x = y$ and there exists some $e\in E(S)$ such that $se = te$ and $x\in D_e^\alpha$. The set of equivalence classes is denoted
\[
\g(\alpha) = \{ [s,x]\mid s\in S, x\in X\}
\]
and becomes a groupoid when given the operations
\[
 d([s,x]) = x,\hspace{0.5cm} r([s,x]) = \alpha_s(x),
\]
\[
[s,x]^{-1} = [s^*, \alpha_s(x)],\hspace{0.5cm}[t,\alpha_s(x)][s,x] = [ts,x].
\]
This is called the {\em groupoid of germs} of $\alpha$. Note that above we are making the identification of the unit space with $X$, because $[e, x] = [f, x]$ for any $e, f\in E(S)$ with $x\in D_e^\alpha, D_f^\alpha$. For $s\in S$ and open set $U\subset D^\alpha_{s^*s}$ we let
\[
\Theta(s, U) = \{[s,x]\mid x\in U\}
\]
and endow $\g(\alpha)$ with the topology generated by such sets. With this topology $\g(\alpha)$ is an \'etale groupoid, sets of the above type are bisections, and if $X$ is totally disconnected $\g(\alpha)$ is ample.

Let $\theta:S\curvearrowright \Et(S)$ be the standard action, and define
\[
\gt(S) = \g(\theta).
\] 
This is called the {\em tight groupoid} of $S$. This was defined first in \cite{Ex08} and studied extensively in \cite{EP16}. 

Let $\g$ be an ample \'etale groupoid, and consider the Boolean inverse monoid $\g^a$. By work of Exel \cite{Ex10} if one uses the above procedure to produce a groupoid from $\g^a$, one ends up with exactly $\g$. In symbols,
\begin{equation}\label{NCstone}
\gt(\g^a) \cong \g\text{ for any ample \'etale groupoid }\g.
\end{equation}
We note this result was also obtained in \cite[Theorem 6.11]{L08} in the case where $\Et(S) = \Eu(S)$. In particular,
\[
\gt(\gt(S)^a) \cong \gt(S) \text{ for all inverse semigroups }S.
\]
This result can be made categorical \cite[Theorem 3.26]{LL13}, and has been generalized to cases where the space of units is not even Hausdorff. This duality between Boolean inverse semigroups and ample \'etale groupoids falls under the broader program of noncommutative Stone duality, see \cite{LL13} for more details. 

\section{C*-algebras of Boolean inverse monoids}

In this section we describe the tight C*-algebra of a general inverse monoid, define the C*-algebra of a Boolean inverse monoid, and show that these two notions coincide for Boolean inverse monoids. 

If $S$ is an inverse monoid, then a {\em representation} of $S$ in a unital C*-algebra $A$ is a map $\pi:S\to A$ such that $\pi(0) = 0$, $\pi(s^*) = \pi(s)^*$, and $\pi(st) = \pi(s)\pi(t)$ for all $s,t\in S$. If $\pi$ is a representation, then $C^*(\pi(E(S)))$ is a commutative C*-algebra. Let 
\[
\mathscr{B}_\pi = \{e\in C^*(\pi(E(S)))\mid e^2 = e = e^*\}
\]
Then this set is a Boolean algebra with operations
\[
e\wedge f = ef, \hspace{1cm}e\vee f = e+f -ef, \hspace{1cm}e^\perp = 1-e.
\]

We will be interested in a subclass of representations of $S$. Take $X,Y\fs E(S)$, and define
\[
E(S)^{X,Y} = \{e\in E(S)\mid e\leqslant x\text{ for all }x\in X, ey = 0 \text{ for all }y\in Y\}
\] 
We say that a representation $\pi: S\to A$ with $A$ unital is {\em tight} if for all $X,Y,Z\fs E(S)$ where $Z$ is a cover of $E(S)^{X,Y}$, we have the equation
\begin{equation}\label{tightequation}
\bigvee_{z\in Z} \pi(z) = \prod_{x\in X}\pi(x)\prod_{y\in Y}(1-\pi(y)).
\end{equation}
The {\em tight C*-algebra} of $S$, denoted $\Ct(S)$, is then the universal unital C*-algebra generated by one element for each element of $S$ subject to the relations that guarantee that the standard map from $S$ to $\Ct(S)$ is tight. The above was all defined in \cite{Ex08} and the interested reader is directed there for the details. It is a fact that $\Ct(S)\cong C^*(\gt(S))$ where the latter is the full groupoid C*-algebra (see e.g. \cite[Theorem 2.4]{Ex10}).

If $S$ has the additional structure of being a Boolean inverse monoid, then we might wonder what extra properties $\pi$ should have, in particular, what is the notion of a ``join'' of two partial isometries in a C*-algebra?

Let $A$ be a C*-algebra, and suppose that $S$ is a Boolean inverse monoid of partial isometries in $A$.  If we have $s, t\in S$ such that $s^*t, st^*\in E(S)$, then
\[
tt^*s = tt^*ss^*s = ss^*tt^*s = s(s^*t)(s^*t)^* = ss^*t
\] 
and if we let $a_{s,t} := s+t - ss^*t = s+t - tt^*s$, this is a partial isometry with range $a_{ss^*,tt^*}$ and support $a_{s^*s,t^*t}$. A short calculation shows that $a_{s,t}$ is the least upper bound for $s$ and $t$ in the natural partial order, and so $a_{s,t} = s\vee t$. It is also straightforward that $r(s\vee t) = rs\vee rt$ for all $r, s, t\in S$. This leads us to the following definitions.

\begin{defn}\label{IMdef}Let $S$ be a Boolean inverse monoid. A {\em Boolean inverse monoid representation} of $S$ in a unital C*-algebra $A$ is a map $\pi: S\to A$ such that 
\begin{enumerate}
\item $\pi(0) = 0$,
\item $\pi(st) = \pi(s)\pi(t)$ for all $s, t\in S$, 
\item $\pi(s^*) = \pi(s)^*$ for all $s\in S$, and
\item $\pi(s\vee t) = \pi(s) + \pi(t) - \pi(ss^*t)$ for all compatible $s, t, \in S$.
\end{enumerate}
\end{defn}

\begin{defn}
Let $S$ be a Boolean inverse monoid. Then the {\em universal C*-algebra of $S$}, denoted $\CB(S)$, is defined to be the universal unital C*-algebra generated by one element for each element of $S$ subject to the relations which say that the standard map of $S$ into $\CB(S)$ is a Boolean inverse monoid representation. The map $\pi_u$ which takes an element $s$ to its corresponding element in $\CB(S)$ will be called the {\em universal Boolean inverse monoid representation of $S$}, and we will sometimes use the notation $\delta_s:=\pi_u(s)$.
\end{defn}
The theory of tight representations was originally developed to deal with representing inverse semigroups (in which joins may not exist) inside C*-algebras, because in a C*-algebra two commuting projections always have a join. It should come as no surprise then that once we are dealing with an inverse semigroup where we \underline{can} take joins, the representations which respect joins end up being exactly the tight representations, see \cite[Corollary 2.3]{DM14}. This is what we prove in the next proposition.
\begin{prop}\label{BIMRtight}
Let $S$ be a Boolean inverse monoid. Then a map $\pi: S \to A$ is a Boolean inverse monoid representation of $S$ if and only if $\pi$ is a tight representation.
\end{prop}
\begin{proof}
Suppose that $\pi$ is a Boolean inverse monoid representation of $S$. Then when restricted to $E(S)$, $\pi$ is a Boolean algebra homomorphism into $\mathscr{B}_\pi$, and so by \cite[Proposition 11.9]{Ex08}, $\pi$ is a tight representation.

On the other hand, suppose that $\pi$ is a tight representation, and first suppose that $e, f\in E(S)$. Then the set $\{e, f\}$ is a cover for $E(S)^{\{e\vee f\},\emptyset}$, so
\[
\pi(e)\vee \pi(f) = \pi(e\vee f).
\]
Now let $s, t\in S$ be compatible, so that $s^*t = t^*s$ and $st^*= ts^*$ are both idempotents, and we have
\[
s^*st^*t = s^*ts^*t = s^*t.
\]
Since $(s\vee t)^*(s\vee t) = s^*s\vee t^*t$, we have
\begin{eqnarray*}
\pi(s\vee t) &=&\pi(s\vee t)\pi(s^*s\vee t^*t)\\ 
             &=&\pi(s\vee t)(\pi(s^*s)+ \pi(t*t) - \pi(s^*st^*t)\\
             &=&\pi(ss^*s\vee ts^*s) + \pi(st^*t\vee tt^*t) - \pi(ss^*st^*t\vee tt^*ts^*s))\\
             &=&\pi(s\vee st^*s) + \pi(ts^*t\vee t) - \pi(st^*t\vee ts^*s)\\
             &=&\pi(s) + \pi(t) - \pi(ss^*t)
\end{eqnarray*}
where the last line follows from the facts that $st^*s\leqslant s$, $ts^*t\leqslant t$ and $ts^*s = st^*t = ss^*t = tt^*s$.
\end{proof}
We have the following consequence of the proof of the above proposition.
\begin{cor}
Let $S$ be a Boolean inverse monoid. Then a map $\pi:S \to A$ is a Boolean inverse monoid representation of $S$ if and only if it is a representation and for all $e, f\in E(S)$ we have $\pi(e\vee f) = \pi(e) + \pi(f) - \pi(ef)$. 
\end{cor}
We now have the following.
\begin{theo}\label{BIMtight}
Let $S$ be a Boolean inverse monoid. Then 
\[
\CB(S) \cong \Ct(S) \cong C^*(\gt(S)).
\]
\end{theo}

In what follows, we will be studying traces on C*-algebras arising from Boolean inverse monoids. However, many of our examples will actually arise from inverse monoids which are not distributive, and so the Boolean inverse monoid in question will actually be $\gt(S)^a$, see \eqref{NCstone}. The map from $S$ to $\gt(S)^a$ defined by 
\[
s\mapsto \Theta(s, D_{s^*s}^\theta)
\]
may fail to be injective, and so we cannot say that a given inverse monoid can be embedded in a Boolean inverse monoid. The obstruction arises from the following situation: suppose $S$ is an inverse semigroup and that we have $e, f\in E(S)$ such that $e\leqslant f$ and for all $0\neq k \leqslant f$ we have $ek \neq 0$, in other words, $\{e\}$ is a cover for $\{f\}^\downarrow$. In such a situation, we say that $e$ is {\em dense} in $f$\footnote{This is the terminology used in \cite[Definition 11.10]{Ex08} and \cite{Ex09}, though in \cite[Section 6.3]{LS14} such an $e$ is called {\em essential} in $f$.}, and by \eqref{tightequation} we must have that $\pi(e) = \pi(f)$ (see also \cite{Ex09} and \cite[Proposition 11.11]{Ex08}). For most of our examples, we will be considering inverse semigroups which have faithful tight representations, though we consider one which does not.

We close this section by recording some consequences of Theorem \ref{BIMtight}. The tight groupoid and tight C*-algebra of an inverse semigroup were extensively studied in \cite{EP16} and \cite{Ste16}, where they gave conditions on $S$ which imply that $\Ct(S)$ is simple and purely infinite. We first recall some definitions from \cite{EP16}.
\begin{defn}
Let $S$ be an inverse semigroup, let $s\in S$ and $e\leqslant s^*s$. Then we say that
\begin{enumerate}
\item $e$ is {\em fixed} by $s$ if $se = e$, and
\item $e$ is {\em weakly fixed} by $s$ if for all $0\neq f\leqslant e$, $fsfs^*\neq 0$.
\end{enumerate}
Denote by $\mathcal{J}_s := \{e\in E(S)\mid se = e\}$ the set of all fixed idempotents for $s\in S$. We note that an inverse semigroup for which $\J_s = \{0\}$ for all $s\notin E(S)$ is called {\em E*-unitary}.
\end{defn}

\begin{theo}
Let $S$ be an inverse semigroup. Then
\begin{enumerate}
\item $\gt(S)$ is Hausdorff if and only if $\J_s$ has a finite cover for all $s\in S$. \cite[Theorem 3.16]{EP16}
\item If $\gt(S)$ is Hausdorff, then $\gt(S)$ is essentially principal if and only if for every $s\in S$ and every $e\in E(S)$ weakly fixed by $s$, there exists a finite cover for $\{e\}$ by fixed idempotents. \cite[Theorem 4.10]{EP16}
\item $\gt(S)$ is minimal if and only if for every nonzero $e, f\in E(S)$, there exist $F\fs S$ such that $\{esfs^*\mid s\in F\}$ is a cover for $\{e\}$.\cite[Theorem 5.5]{EP16}
\end{enumerate}
\end{theo}
We translate the above to the case where $S$ is a Boolean inverse monoid.

\begin{prop}
Let $S$ be a Boolean inverse monoid. Then
\begin{enumerate}
\item $\gt(S)$ is Hausdorff if and only if for all $s\in S$, there exists an idempotent $e_s$ with $se_s = e_s$ such that if $e$ is fixed by $s$, then $e\leqslant e_s$.
\item If $\gt(S)$ is Hausdorff, then $\gt(S)$ is essentially principal if and only if for every $s\in S$, $e$ weakly fixed by $s$ implies $e$ is fixed by $s$.
\item $\gt(S)$ is minimal if and only if for every nonzero $e, f\in E(S)$, there exist $F\fs S$ such that $e\leqslant \bigvee_{s\in F}sfs^*$.
\end{enumerate}
\end{prop}
\begin{proof}
Statements 2 and 3 are easy consequences of taking the joins of the finite covers mentioned. Statement 1 is central to what follows, and is proven in Lemma \ref{glblemma}.
\end{proof}

If an \'etale groupoid $\g$ is Hausdorff, then $C^*(\g)$ is simple if and only if $\g$ is essentially principal, minimal, and satisfies weak containment, see \cite{BCFS14} (also see \cite{ES15} for a discussion of amenability of groupoids associated to inverse semigroups).

%\section{A Boolean inverse monoid from an inverse semigroup}
%In this section we show that given an inverse semigroup $S$ arising from a C*-algebra, one may produce a Boolean inverse monoid inside $\gt(S)^a$ whose C*-algebra is isomorphic to the tight C*-algebra of $S$.

%Let $S$ be a RuyH inverse semigroup with tight groupoid $\gt(S)$. Consider the map $\sigma: S\to \gt(S)^a$ defined by  $$\sigma(s)= \Theta(s, D_{s^*s}) = \{[s,\xi]\in \gt(S) \mid \xi\in D_{s^*s}\}.$$  

%By \cite[Proposition 3.7]{Ste10}, $\gt(S)^a$ is closed under relative complements and pairwise intersections, and is also {\em orthogonally complete} in that if $U,V\in \gt(S)^a$ with $UV = \emptyset = VU$, then $U\cup V\in \gt(S)^a$. 

%Let $S^\vee$ be the inverse semigroup of all joins of finite compatible sets of $\sigma(S)$ inside $\gt(S)^a$. Notice that $E(\gt(S)^a)$ is a Boolean algebra. Let 
%\[
%\mathcal{B}(S) = \text{Inverse semigroup generated by }S^\vee\text{ and } E(\gt(S)^a).
%\]
\section{Invariant means and traces}
In this section we consider invariant means on Boolean inverse monoids, and show that such functions always give rise to traces on the associated C*-algebras. This definition is from \cite{KLLR15}.
\begin{defn}
Let $S$ be a Boolean inverse monoid. A nonzero function $\mu: E(S) \to [0,\infty)$ will be called an {\em invariant mean} if 
\begin{enumerate}
\item $\mu(s^*s) = \mu(ss^*)$ for all $s\in S$
\item $\mu(e\vee f) = \mu(e) + \mu(f)$ for all $e, f\in E(S)$ such that $ef = 0$.
\end{enumerate}
If in addition $\mu(1) = 1$, we call $\mu$ a {\em normalized invariant mean}. An invariant mean $\mu$ will be called {\em faithful} if $\mu(e) = 0$ implies $e = 0$. We will denote by $M(S)$ the affine space of all normalized invariant means on $S$. 
\end{defn}

We make an important assumption on the Boolean inverse monoids we consider here. This assumption is equivalent to the groupoid $\gt(S)$ being Hausdorff \cite[Theorem 3.16]{EP16}.\footnote{In \cite{StLCM}, we define condition (H) for another class of semigroups, namely the right LCM semigroups. Right LCM semigroups and inverse semigroups are related, but the intersection of their classes is empty (because right LCM semigroups are left cancellative and we assume that our inverse semigroups have a zero element). We note that a right LCM semigroup $P$ satisfies condition (H) in the sense of \cite{StLCM} if and only if its left inverse hull $I_l(P)$ satisfies condition \eqref{condH} in the sense of the above.}
\begin{equation}\label{condH}\tag{H}
\text{For every }s\in S,\text{ the set }\J_s = \{e\in E(S)\mid se = e\}\text{ admits a finite cover.}
\end{equation}

The next lemma records straightforward consequences of condition \eqref{condH} when $S$ happens to be a Boolean inverse monoid.

\begin{lem}\label{glblemma}
Let $S$ be Boolean inverse monoid which satisfies condition \eqref{condH}. Then,
\begin{enumerate}
\item  for each $s\in S$ there is an idempotent $e_s$ such that for any finite cover $C$ of $\J_s$, 
\begin{equation}\label{esdef}
e_s = \bigvee_{c\in C}c.
\end{equation}
 and $\J_s = \J_{e_s}$,
\item $e_{s^*} = e_s$ for all $s\in S$,
\item $e_{st}\leqslant ss^*, t^*t$ for all $s, t\in S$, and 
\item $e_{s^*t}e_{t^*r}\leqslant e_{s^*r}$ for all $s,t,r\in S$.
\end{enumerate}
\end{lem}

\begin{proof}
To show the first statement, we need to show that any two covers give the same join. If $\J_s = \{0\}$, there is nothing to do. So suppose that $0 \neq e\in \J_s$, suppose that $C$ is a cover for $\J_s$, and let $e_C = \bigvee_{c\in C}c$. Indeed, the element $ee_C^\perp$ must be in $\J_s$, and since it is orthogonal to all elements of $C$ and $C$ is a cover, $ee_C^\perp$ must be 0. Hence we have
\[
e = ee_C\vee ee_C^\perp = ee_C
\]
and so $e\leqslant e_C$. Now if $K$ is another cover for $\J_s$  with join $e_K$ and $k\in K$, we must have that $k\leqslant e_C$, and so $e_K\leqslant e_C$. Since the argument is symmetric, we have proven the first statement.

To prove the second statement, if $e\in \J_s$ then we have
\[
ses^* = es^* = (se)^* = e
\]
and so 
\[
s^*e = s^*(ses^*) = es^*ss^* = es^* = (se)^* = e
\]
and again by symmetry we have $\J_s = \J_{s^*}$ and so $e_s = e_{s^*}$.

To prove the third statement, we notice
\[
ss^*e_{st} = ss^*ste_{st} = ste_{st} = e_{st}
\]
\[
e_{st}t^*t = ste_{st}t^*t = stt^*te_{st} = ste_{st} = e_{st}.
\]

For the fourth statement, we calculate (using 2)
\begin{eqnarray*}
e_{s^*t}e_{t^*r} &=& s^*te_{s^*t}e_{t^*r} = s^*tt^*re_{s^*t}e_{t^*r}\\
 &=& s^*tt^*rr^*te_{s^*t}e_{t^*r} = s^*rr^*te_{s^*t}e_{t^*r} \\
 &=& s^*r e_{s^*t}e_{t^*r}
\end{eqnarray*}
hence $e_{s^*t}e_{t^*r} \leqslant s^*r$ and so $e_{s^*t}e_{t^*r} \leqslant e_{s^*r}$.
\end{proof}
In what will be a crucial step to obtaining a trace from an invariant mean, we now obtain a relationship between $e_{st}$ and $e_{ts}$.
\begin{lem}\label{estets}
Let $S$ be Boolean inverse monoid which satisfies condition \eqref{condH}. Then for all $s, t\in S$, we have that $s^*e_{st}s = e_{ts}$.
\end{lem}

\begin{proof}
Suppose that $e\in \J_{ts}$. Then $tse = e$, and so 
\[
(st)ses^* = ses^* 
\]
hence $ses^*\in \J_{st}$. If $C$ is a cover of $\J_{st}$ and $f\in \J_{ts}$, there must exist $c\in C$ such that $c(sfs^*)\neq 0$. Hence
\begin{eqnarray*}
css^*sfs^* &\neq& 0\\
ss^*csfs^* &\neq& 0\\
s^*csf &\neq &0
\end{eqnarray*}
and so we see that $s^*Cs$ is a cover for $\J_{ts}$. By Lemma \ref{glblemma},
\[
e_{ts} = \bigvee_{c\in C}s^*cs = s^*\left(\bigvee_{c\in C}c\right)s = s^*e_{st}s.
\]
\end{proof}
Lemma \ref{estets} and Lemma \ref{glblemma}.3 imply that for all $s,t\in S$ and all $\mu\in M(S)$, we have $\mu(e_{st}) = \mu(e_{ts})$. 
\begin{rmk}
We are thankful to Ganna Kudryavtseva for pointing out to us that the proofs Lemmas \ref{glblemma} and \ref{estets} can be simplified by using the fact from \cite[Theorem 8.20]{KL14} that a Boolean inverse monoid $S$ satisfies condition \eqref{condH} if and only if every pair of elements in $S$ has a meet (see also \cite[Proposition 3.7]{Ste10} for another wording of this fact). From this, one can see that for all $s\in S$ we have
\[
e_s = s\wedge(s^*s) = s\wedge (ss^*).
\]
\end{rmk}
\begin{defn}
Let $A$ be a C*-algebra. A bounded linear functional $\tau:A\to \CC$ is called a {\em trace} if 
\begin{enumerate}
\item $\tau(a^*a) \geq 0$ for all $a\in A$,
\item $\tau(ab) = \tau(ba)$ for all $a, b\in A$.
\end{enumerate}
A trace $\tau$ is said to be {\em faithful} if $\tau(a^*a) >0$ for all $a\neq 0$. A trace $\tau$ on a unital C*-algebra is called a {\em tracial state} if $\tau(1) = 1$. The set of all tracial states of a C*-algebra $A$ is denoted $T(A)$.
\end{defn}

We are now able to define a trace on $\CB(S)$ for each $\mu\in M(S)$. 

\begin{prop}\label{taumu}
Let $S$ be Boolean inverse monoid which satisfies condition \eqref{condH}, and let $\mu\in M(S)$. Then there is a trace $\tau_\mu$ on $\CB(S)$ such that 
\[
\tau_\mu(\delta_s) = \mu(e_s)\hspace{1cm} \text{for all }s\in S.
\]
If $\mu$ is faithful, then the restriction of $\tau_\mu$ to $\Cr(\gt(S))$ is a faithful trace.
\end{prop} 
\begin{proof}
We define $\tau_\mu$ to be as above on the generators $\delta_s$ of $\CB(S)$, and extend it to $B:= $span$\{\delta_s\mid s\in S\}$, a dense $*$-subalgebra of $\CB(S)$. 

We first show that $\tau_\mu(\delta_s\delta_t) = \tau_\mu(\delta_t\delta_s)$. Indeed, by Lemmas \ref{glblemma} and \ref{estets}, we have
\begin{eqnarray*}
\tau_\mu(\delta_s\delta_t)&=& \mu(e_{st}) = \mu(e_{st}ss^*) = \mu(e_{st}ss^*e_{st}) = \mu((e_{st}s)(e_{st}s)^*)\\
& =& \mu((e_{st}s)^*(e_{st}s)) = \mu(s^*e_{st}s) = \mu(e_{ts}) = \tau_\mu(\delta_t\delta_s).
\end{eqnarray*}
Since $\tau_\mu$ is extended linearly to $B$, we have that $\tau_\mu(ab) = \tau_\mu(ba)$ for all $a, b\in B$. 

Let $F$ be a finite index set and take $x = \sum_{i\in F}a_i\delta_{s_i}$ in $B$. We will show that $\tau_\mu(x^*x) \geq 0$. For $i, j\in F$, we let $e_{ij} = e_{s_i^*s_j}$ and note that $e_{ij} = e_{ji}$. We calculate:

\begin{eqnarray*}
x^*x &=& \left(\sum_{s\in S}\overline{a_i}\delta_{s_i^*}\right)\left(\sum_{j\in F}a_j\delta_{s_j}\right)\\
     &=& \sum_{i, j\in F}\overline{a_i}a_j\delta_{s_i^*s_j}\\
    % &=& \sum_{s\neq	t} |a_s|^2 \delta_{s^*s} + |a_t|^2 \delta_{t^*t} + \overline{a_s}a_t \delta_{s^*t} + \overline{a_t}a_s \delta_{t^*s}\\
\tau_\mu(x^*x)    &=& \sum_{i,j\in F} \overline{a_i}a_j \mu(e_{ij})\\
   &=& \sum_{i\in F} |a_i|^2\mu(e_{ii}) + \sum_{i, j\in F, i\neq j}(\overline{a_i}a_j +\overline{a_j}a_i)\mu(e_{ij}).
\end{eqnarray*}
We will show that this sum is positive by using an orthogonal decomposition of the $e_{ij}$. Let $F^2_{\neq} = \{ \{i,j\}\subset F \mid i\neq j\}$, and let $D(F^2_{\neq})=\{(A,B)\mid A\cup B = F^2_{\neq}, A\cap B = \emptyset\}$. For $a = \{i, j\}\in F^2_{\neq}$, let $e_a = e_{ij}$. We have
\[
e_{ij} = e_{ij} \bigvee_{(A, B)\in D(F^2_{\neq})} \left(\prod_{a\in A, b\in B}e_ae_b^\perp\right)
\]
where the join is an orthogonal join. Of course, the above is only nonzero when $\{i,j\}\in A$. We also notice that
\[
e_{ii} \geqslant \bigvee_{\substack{(A, B)\in D(F^2_{\neq})\\ i\in \cup A}} \left(\prod_{a\in A, b\in B}e_ae_b^\perp\right)
\]
and so $\tau_\mu(x^*x)$ is larger than a linear combination of terms of the form $\mu\left(\prod_{a\in A, b\in B}e_ae_b^\perp\right)$ for partitions $(A,B)$ of $F^2_{\neq}$: specifically, $\tau_\mu(x^*x)$ is greater than or equal to 
\begin{equation}\label{tracesum}
\sum_{(A, B)\in D(F^2_{\neq})}\left[\left(\sum_{i\in \cup A}|a_i|^2 + \sum_{a = \{j,k\}\in A}\left(\overline{a_i}a_j +\overline{a_j}a_i\right)\right)\mu\left(\prod_{a\in A, b\in B}e_ae_b^\perp\right)\right]
\end{equation}
If a term $\prod_{a\in A, b\in B}e_ae_b^\perp$ is not zero, then we claim that the relation
\[
i \sim j \text{ if and only if } i = j \text{ or } \{i, j\} \in A
\]
is an equivalence relation on $\cup A$. Indeed, suppose that $i, j, k\in \cup A$ are all pairwise nonequal and $\{i,j\}, \{j,k\}\in A$. By Lemma \ref{glblemma}.4, $e_{ij}e_{jk}\leqslant e_{ik}$ and since the product is nonzero, we must have that $\{i,k\}\in A$. Writing $[\cup A]$ for the set of equivalence classes, we have 
\begin{eqnarray*}
\sum_{i\in \cup A}|a_i|^2 + \sum_{a = \{j,k\}\in A}\left(\overline{a_i}a_j +\overline{a_j}a_i\right) &=& \sum_{C\in [\cup A]} \left(\sum_{i\in C}|a_i|^2 + \sum_{\substack{i,j\in C\\ i\neq j}} \left(\overline{a_i}a_j +\overline{a_j}a_i\right) \right)\\
 &=& \sum_{C\in [\cup A]} \left|\sum_{i\in C}a_i  \right|^2.
\end{eqnarray*}
Hence, $\tau_\mu(x^*x)\geq 0$, and $\tau_\mu$ is positive on $B$. Hence, $\tau_\mu$ extends to a trace on $\CB(S)$. 

The above calculation shows that if $\mu$ is faithful, then $\tau_\mu$ is faithful on $B$. A short calculation shows that $E(\delta_s) = \delta_{e_s}$, where $E$ is as in \eqref{conditionalexpectation}. Furthermore, it is clear that on $B$ we have that $\tau_\mu = \tau_\mu\circ E$, and so we will show that $\tau_\mu$ is faithful on $\Cr(\gt(S))$ if we show that $\tau_\mu(a)> 0$ for all nonzero positive $a\in  C(\Et(S))$. If $a\in C(\Et(S))$ is positive, then it is bounded above zero on some clopen set given by $D_e$ for some $e\in E(S)$. Hence, $\tau_\mu(a)\geq \tau_\mu(\delta_e) = \mu(e)$ which must be strictly positive because $\mu$ is faithful.
\end{proof}
We now show that given a trace on $\CB(S)$ we can construct an invariant mean on $S$.
\begin{prop}\label{mutau}
Let $S$ be Boolean inverse monoid, let $\pi_u: S \to \CB(S)$ be the universal Boolean monoid representation of $S$, and take $\tau\in T(\CB(S))$. Then the map $\mu_\tau : E(S) \to [0,\infty)$ defined by 
\[
\mu_\tau(e) = \tau(\pi_u(e)) = \tau(\delta_e)
\]
is a normalized invariant mean on $S$. If $\tau$ is faithful then so is $\mu_\tau$.
\end{prop}
\begin{proof}
That $\mu_\tau$ takes positive values follows from $\tau$ being positive. We have
\begin{eqnarray*}
\mu_\tau(s^*s) &=& \tau(\pi_u(s^*s)) = \tau(\pi_u(s^*)\pi_u(s))\\
 & =& \tau(\pi_u(s)\pi_u(s^*)) = \tau(\pi_u(ss^*)) \\
&=& \mu_\tau(ss^*).
\end{eqnarray*}
Also, if $e, f\in E(S)$ with $ef = 0$, then
\begin{eqnarray*}
\mu_\tau(e\vee f) &=& \tau(\pi_u(e\vee f)) = \tau(\pi_u(e) + \pi_u(f))\\
 & =& \tau(\pi_u(e)) + \tau(\pi_u(f)) \\
 &=& \mu_\tau(e) + \mu_\tau(f).
\end{eqnarray*}
If $\tau$ is faithful and $e\neq 0$, $\tau(\delta_e)> 0$ because $\delta_e$ is positive and nonzero, and so $\mu_\tau$ is faithful.
\end{proof}
\begin{prop}
Let $S$ be Boolean inverse monoid which satisfies condition \eqref{condH}. Then the map 
\[
\mu\mapsto \tau_\mu\mapsto \mu_{\tau_\mu}
\]
is the identity on $M(S)$.
\end{prop}
\begin{proof}
This is immediate, since if $\mu\in M(S)$ and $e\in E(S)$ we have
\[
\mu_{\tau_\mu}(e) = \tau_\mu(\pi_u(e))= \tau_\mu(\delta_e) = \mu(e).
\]
\end{proof}
Given the above, one might wonder under which circumstances we have that $T(\CB(S)) \cong M(S)$. This is not true in the general situation -- take for example $S$ to be the group $\ZZ_2 = \{1, -1\}$ with a zero element adjoined -- this is a Boolean inverse monoid. Here $M(S)$ consists of one element, namely the function which takes the value 1 on $1$ and the value $0$ on the zero element. The C*-algebra of $S$ is the group C*-algebra of $\ZZ_2$, which is isomorphic to $\CC^2$, a C*-algebra with many traces (taking the dot product of an element of $\CC^2$ with any nonnegative vector whose entries add to 1 determines a normalized trace on $\CC^2$). 

One can still obtain this isomorphism using the following.

\begin{defn}
Let $\g$ be an \'etale groupoid. A regular Borel probability measure $\nu$ on $\g^{(0)}$ is called {\em $\g$-invariant} if for every bisection $U$ one has that $\nu(r(U)) = \nu(d(U))$. The affine space of all regular $\g$-invariant Borel probability measures is denoted $IM(\g)$.
\end{defn}

The following is a special case of \cite[Proposition 3.2]{KR06}.

\begin{theo}\label{trace-measure}(cf \cite[Proposition 3.2]{KR06})
Let $\g$ be a Hausdorff principal \'etale groupoid with compact unit space. Then
\[
T(C^*_{\text{red}}(\g)) \cong IM(\g)
\]
For $\tau\in T(C^*_{\text{red}}(\g))$ the image of $\tau$ under the above isomorphism is the regular Borel probability measure $\nu$ whose existence is guaranteed by the Riesz representation theorem applied to the positive linear functional on $C(\g^{(0)})$ given by restricting $\tau$.
\end{theo}
For a proof of Theorem \ref{trace-measure} in the above form, see \cite[Theorem 3.4.5]{PNotes}.

For us, the groupoid $\gt(S)$ satisfies all of the conditions in Theorem \ref{trace-measure}, except possibly for being principal. Also note that in the general case, $C^*_{\text{red}}(\gt)$ may not be isomorphic to $\CB(S)$. So if we restrict our attention to Boolean inverse monoids which have principal tight groupoids and for which $C^*_{\text{red}}(\gt(S))\cong \CB(S)$ (that is to say, Boolean inverse monoids for which $\gt(S)$ satisfies weak containment), we can obtain the desired isomorphism. While this may seem like a restrictive set of assumptions, they are all satisfied for the examples we consider here.

\begin{prop}\label{etanu}
Let $S$ be Boolean inverse monoid which satisfies condition \eqref{condH}, and suppose $\nu\in IM(\gt(S))$. Then the map $\eta_\nu: E(S)\to [0,\infty)$ defined by
\[
\eta_\nu(e) = \nu(D_e^\theta)
\]
is a normalized invariant mean on $S$. The map that sends $\nu\mapsto \eta_\nu$ is an affine isomorphism of $IM(\gt(S))$ and $M(S)$.
\end{prop}
\begin{proof}
That $\eta_\nu(s^*s) = \eta_\nu(ss^*)$ follows from invariance of $\nu$ applied to the bisection $\Theta(s, D_{s^*s})$, and that $\eta_\nu$ is additive over orthogonal joins follows from the fact that $\nu$ is a measure. This map is clearly affine. Suppose that $\eta_\nu = \eta_\kappa$ for $\nu,\kappa\in IM(\gt(S))$. Then $\nu,\kappa$ agree on all sets of the form $D_e^\theta$, and since these sets generate the topology on $\Et(S)$, $\nu$ and $\kappa$ agree on all open sets. Since they are regular Borel probability measures they must be equal, and so $\nu\mapsto\eta_\nu$ is injective.

To get surjectivity, let $\mu$ be an invariant mean, and let $\tau_\mu$ be as in Proposition \ref{taumu}. Then restricting $\tau_\mu$ to $C(\Et(S))$ and invoking the Riesz representation theorem gives us a regular invariant probability measure $\nu$ on $\Et(S)$, and we must have $\eta_\nu = \mu$.
\end{proof}
\begin{cor}
Let $\g$ be an ample Hausdorff groupoid. Then $IM(\g)\cong M(\g^a)$.
\end{cor}
So the invariant means on the ample semigroup of an ample Hausdorff groupoid are in one-to-one correspondence with the $\g$-invariant measures.
\begin{theo}\label{maintheo}
Let $S$ be Boolean inverse monoid which satisfies condition \eqref{condH}. Suppose that $\gt(S)$ is principal, and that $C^*_{\text{red}}(\gt(S)) \cong \CB(S)$. Then
\[
T(\CB(S))\cong M(S)
\]
via the map which sends $\tau$ to $\mu_\tau$ as in Proposition \ref{mutau}. In addition, both are isomorphic to $IM(\gt(S))$.
\end{theo}
\begin{proof}
This follows from Theorem \ref{trace-measure} and Proposition \ref{etanu}.
%Suppose that $\mu\in M(S)$, and let $\tau_\mu\in T(C^*(S))$ be as in Proposition \ref{taumu}. Now let $\nu_{\tau_\mu}$ be the $\gt(S)$-invariant measure from $\tau_\mu$ guaranteed by Theorem \ref{trace-measure}. We have the sequence of maps
%\begin{equation}\label{quadcomposition}\begin{array}{ccccccc}
%M(S)&\to&T(C^*(S))&\to&IM(\gt(S))&\to & M(S)\\
%\mu&\mapsto&\tau_\mu&\mapsto&\nu_{\tau_\mu}&\mapsto &\eta_{\nu_{\tau_\mu}}.
%\end{array}
%\end{equation}
%The ridiculous quadruple subscript composition of all these maps is easily seen to be injective, and the middle map is an isomorphism, so they must all be isomorphisms.
\end{proof}

%We also have the following, which applies when $\gt(S)$ may not be principal.
%\begin{theo}
%Let $S$ be Boolean inverse monoid which satisfies condition \eqref{condH}. Then $M(S)\cong IM(\gt(S))$.
%\end{theo}
%\begin{proof}
%If $\mu\in M(S)$, then the map in \eqref{quadcomposition} is still seen to be an isomorphism

%\end{proof}

There are many results in the literature concerning traces which now apply to our situation.  

\begin{cor}
Let $S$ be Boolean inverse monoid which satisfies condition \eqref{condH}. If $S$ admits a faithful invariant mean, then $C^*_{\text{red}}(\gt(S))$ is stably finite. If in addition $\gt(S)$ satisfies weak containment, $\CB(S)$ is stably finite.
\end{cor}
\begin{proof}
If $\mu$ is a faithful invariant mean, then after normalizing one obtains a faithful trace on $\Cr(\gt(S))$ by Proposition \ref{taumu}. Now the result is standard, see for example \cite[Exercise 5.2]{LLRBook}. 
\end{proof}

\begin{cor}
Let $S$ be Boolean inverse monoid which satisfies condition \eqref{condH}. If $\CB(S)$ is stably finite and exact, then $S$ has an invariant mean.
\end{cor}
\begin{proof}
This is a consequence of the celebrated result of Haagerup \cite{Ha14} when applied to Proposition \ref{mutau}
\end{proof}
For the undefined terms above, we direct the interested reader to \cite{BO08}. We also note that exactness of $\CB(S)$ has recently been considered in \cite{Li16} and \cite{AD16a}.
\section{Examples}\label{examplessection}

\subsection{AF inverse monoids}

This is a class of Boolean inverse monoids introduced in \cite{LS14} motivated by the construction of AF C*-algebras from Bratteli diagrams. 

A {\em Bratteli diagram} is an infinite directed graph $B = (V, E, r, s)$ such that 

\begin{enumerate}
\item $V$ can be written as a disjoint union of finite sets $V = \cup_{n\geq 0}V_n$
\item $V_0$ consists of one element $v_0$, called the {\em root},
\item for all edges $e\in E$, $s(e)\in V_i$ implies that $r(e)\in V_{i+1}$ for all $i \geq 0$, and
\item for all $i\geq 1$ and all $v\in V_i$, both $r^{-1}(v)$ and $s^{-1}(v)$ are finite and nonempty.
\end{enumerate}
We also denote $s^{-1}(V_i) := E_i$, so that $E = \cup_{n\geq 0}E_n$. Let $E^*$ be the set of all finite paths in $B$, including the vertices (treated as paths of length zero). For $v, w\in V\cup E$, let $vE^*$ denote all the paths starting with $v$, let $E^*w$ be all the paths ending with $w$, and let $vE^*w$ be all the paths starting with $v$ and ending with $w$.

Given a Bratteli diagram $B = (V, E, r, s)$ we construct a C*-algebra as follows. We let 
\[
A_0 = \CC
\]
\[
A_1 = \bigoplus_{v\in V_1}\M_{|r^{-1}(v)|},
\]
and define $k_1(v) = |r^{-1}(v)|$ for all $v\in V_1$. For an integer $i>1$ and $v\in V_i$, let 
\begin{equation}\label{ki}
k_i(v) = \sum_{\gamma\in r^{-1}(v)}k_{i-1}(s(\gamma)).
\end{equation} 
Define
\[
A_{i} = \bigoplus_{v\in V_i}\M_{k_{i}(v)}
\]
Now for all $i\geq 0$, one can embed $A_i\hookrightarrow A_{i+1}$ by viewing, for each $v\in V_{i+1}$
\[
\bigoplus_{\gamma\in r^{-1}(v)}\M_{k_{i}(s(\gamma))} \subset \M_{k_{i+1}(v)}
\]
where the algebras in the direct sum are orthogonal summands along the diagonal in $\M_{k_{i+1}(v)}$. So $A_0\hookrightarrow A_1\hookrightarrow A_1\hookrightarrow \cdots$ can be viewed as an increasing union of finite dimensional C*-algebra, all of which can be realized as subalgebras of $\mathcal{B}(\mathcal{H})$ for the same $\mathcal{H}$, and so we can form the norm closure of the union
\[
A_B:= \overline{\bigcup_{n\geq 0}A_n}.
\]
This C*-algebra is what is known as an {\em AF algebra}, and every unital AF algebra arises this way from some Bratteli diagram.

The AF algebra $A_B$ can always be described as the C*-algebra of a principal groupoid derived from $B$, see \cite{R80} and \cite{ER06}. We reproduce this construction here. Let $X_B$ denote the set of all infinite paths in $B$ which start at the root. When given the product topology from the discrete topologies on the $E_n$, this is a compact Hausdorff totally disconnected space. For $\alpha\in v_0E^*$, we let $C(\alpha) = \{x\in X_B\mid x_i = \alpha_i \text{ for all }i = 0, \dots, |\alpha|-1\}$. Sets of this form are clopen and form a basis for the topology on $X_B$. For $n\in \NN$, let
\[
\mathcal{R}^{(n)}_B = \{(x, y)\in X\times X\mid x_i = y_i \text{ for all } i\geq n+1\}
\]
so a pair of infinite paths $(x,y)$ is in $\mathcal{R}^{(n)}_B$ if and only if $x$ and $y$ agree after the vertices on level $n$. Clearly, $\mathcal{R}^{(n)}_B \subset \mathcal{R}^{(n+1)}_B$, and so we can form their union
\[
\mathcal{R}_B = \bigcup_{n\in \NN}\mathcal{R}^{(n)}_B.
\]
This is an equivalence relation, known as {\em tail equivalence} on $X_B$. For $v\in V\setminus\{v_0\}$ and $\alpha, \beta\in v_0E^*v$, define
\[
C(\alpha, \beta) = \{(x, y)\in \mathcal{R}_B\mid x\in C(\alpha), y\in C(\beta)\}
\]
sets of this type form a basis for a topology on $\mathcal{R}_B$, and with this topology $\mathcal{R}_B$ is a principal Hausdorff \'etale groupoid with unit space identified with $X_B$, and
\[
C^*(\mathcal{R}_B)\cong C_{\text{red}}^*(\mathcal{R}_B) \cong A_B.
\]

In \cite{LS14}, a Boolean inverse monoid is constructed from a Bratteli diagram, mirroring the above construction. We will present this Boolean inverse monoid in a slightly different way which may be enlightening. Let $B = (V, E, r, s)$ be a Bratteli diagram. Let $S_0$ be the Boolean inverse monoid (in fact, Boolean algebra) $\{0,1\}$. For each $i \geq 1$, let
\[
S_i = \bigoplus_{v\in V_i}\I(v_0E^*v)
\]
where as in Section \ref{ISGsubsection}, $\I(X)$ denotes the set of partially defined bijections on $X$.  

If $v\in V_{i+1}$ and $\gamma \in r^{-1}(v)$ then one can view $\I(v_0E^*\gamma)$ as a subset of $\I(v_0E^*v)$, and if $\eta\in r^{-1}(v)$ with $\gamma\neq \eta$, $\I(v_0E^*\gamma)$ and $\I(v_0E^*\eta)$ are orthogonal. Furthermore, $\I(v_0E^*\gamma)$ can be identified with $\I(v_0E^*s(\gamma))$ Hence the direct sum over $r^{-1}(v)$ can be embedded into $\I(v_0E^*v)$:
\begin{equation}\label{subISG}
\bigoplus_{\gamma\in r^{-1}(v)}\I(v_0E^*s(\gamma)) \hookrightarrow \I(v_0E^*v).
\end{equation}
This allows us to embed $S_i\hookrightarrow S_{i+1}$
\[
\bigoplus_{v\in V_i}\I(v_0E^*v) \hookrightarrow \bigoplus_{w\in V_{i+1}}\I(v_0E^*w)
\]
where an element $\phi$ in a summand $\I(v_0E^*v)$ gets sent to $|s^{-1}(v)|$ summands on the right, one for each $\gamma\in s^{-1}(v)$: $\phi$ will be sent to the summand inside $\I(v_0E^*s(\gamma))$ corresponding to $v$ in left hand side of the embedding from \eqref{subISG}. We then define
\[
I(B) = \lim_{\to}(S_{i}\hookrightarrow S_{i+1})
\]
This is a Boolean inverse monoid \cite[Lemma 3.13]{LS14}. As a set $I(B)$ is the union of all the $S_i$, viewed as an increasing union via the identifications above. In \cite[Remark 6.5]{LS14}, it is stated that the groupoid one obtains from $I(B)$ (i.e., $\gt(I(B))$) is exactly tail equivalence. We provide the details of that informal discussion here. 

We will describe the ultrafilters in $E(I(B))$, a Boolean algebra. For $v\in V_i$ and a path $\alpha \in v_0E^*v$, let $e_\alpha =$ Id$_{\{\alpha\}}\in \I(v_0E^*v)$. As $v$ ranges over all of $V_i$ and $\alpha$ ranges over all of $v_0E^*v$, these idempotents form a orthogonal decomposition of the identity of $I(B)$. Hence, given an ultrafilter $\xi$ and $i> 0$ there exists one and only one path, say $\alpha^{(i)}_\xi$ ending at level $i$ with $e_{\alpha^{(i)}_\xi}\in \xi$. Furthermore, if $j> i$, we must have that $\alpha^{(i)}_\xi$ is a prefix of $\alpha^{(j)}_\xi$, because products in an ultrafilter cannot be zero. So for $x\in X_B$, if we define
\[
\xi_x = \{e_{\alpha}\mid \alpha\text{ is a prefix of }x\}
\]
then we have that 
\[
\Eu(I(B)) = \{\xi_x\mid x\in X_B\}
\]
By \cite[Proposition 2.6]{EP16}, the set
\[
\{U(\{e_\alpha\}, \emptyset)\mid \alpha\text{ is a prefix of }x\}
\]
is a neighbourhood basis for $\xi_x$. The map $\lambda:X_B\to \Eu(I(B))$ given by $\lambda(x) = \xi_x$ is a bijection, and since $U(\{e_\alpha\}, \emptyset) = \lambda(C(\alpha))$, it is a homeomorphism. If $\phi\in S_i$ such that $\phi^*\phi\in \xi_x$, then we must have that one component of $\phi$ is in $\I(v_0E^*r(x_i))$, and we must have that 
\begin{equation}\label{bratteliaction}
\theta_\phi(\xi_x) = \xi_{\phi(x_0x_1\dots x_i)x_{i+1}x_{i+2}\dots} 
\end{equation}

Finally, we claim that $\mathcal{R}_B$ is isomorphic to $\gt(I(B))$. We define a map
\[
\Phi: \gt(I(B)) \to \mathcal{R}_B
\]
\[
\Phi([\phi, \xi_x]) \mapsto (\phi(x_0x_1\dots x_i)x_{i+1}x_{i+2}\dots, x)
\]
where $\phi$ and $x$ are as in \eqref{bratteliaction}. If $\Phi([\phi, \xi_x]) = \Phi([\psi, \xi_y])$, then clearly we must have $\xi_x = \xi_y$. We must also have that $\phi,\psi\in S_i$, and $\phi e_{x_0x_1\dots x_i} = \psi e_{x_0x_1\dots x_i}$, hence $[\phi, \xi_x] = [\psi, \xi_y]$. It is straightforward to verify that $\Phi$ is surjective and bicontinuous, and so $\mathcal{R}_B\cong \gt(I(B))$. Since they are both \'etale, their C*-algebras must be isomorphic. Hence with the above discussion, we have proven the following.
\begin{theo}\label{AFBIM}
Let $B$ be a Bratteli diagram. Then 
\[
\CB(I(B))\cong A_B.
\]
Furthermore, every unital AF algebra is isomorphic to the universal C*-algebra of a Boolean inverse monoid of the form $I(B)$ for some $B$.
\end{theo}
Recall that a compact convex metrizable subset $X$ of a locally convex space is a {\em  Choquet simplex} if and only if for each $x\in X$ there exists a unique measure $\nu$ concentrated on the extreme points of $X$ for which $x$ is the center of gravity of $X$ for $\nu$ \cite{Ph01}. Now we can use the following seminal result of Blackadar to make a statement about the set of normalized invariant means for AF inverse monoids.

\begin{theo}({\bf Blackadar}, see \cite[Theorem 3.10]{Bl80}) Let $\Delta$ be any metrizable Choquet simplex. Then there exists a unital simple AF algebra $A$ such that $T(A)$ is affinely isomorphic to $\Delta$.
\end{theo}

\begin{cor}
Let $\Delta$ be any metrizable Choquet simplex. Then there exists an AF inverse monoid $S$ such that $M(S)$ is affinely isomorphic to $\Delta$. 
\end{cor}
\begin{proof}
This result follows from Theorem \ref{maintheo} because $\gt(S)$ is Hausdorff, amenable, and principal for every AF inverse monoid $S$. 
\end{proof}
%A C*-algebra $A$ is said to be an {\em AF algebra} if $A$ can be written as the closure of an increasing union of finite dimensional C*-algebras, i.e. there exist finite dimensional C*-algebras $A_i\subset A$ such that $A_i\subset A_{i+1}$ for all $i\geq 0$ and 
%\[
%\overline{\bigcup_{n\geq 0}A_n} = A.
%\]
%Here we will only consider unital AF algebras and so require that the embeddings $A_i\subset A_{i+1}$ are unital. 

%From a unital AF algebra we construct a Bratteli diagram as follows. Since $A_n$ is finite dimensional, it must be a direct sum of full matrix algebras
%\[
%A_n = \M_{k_n(1)}\oplus \M_{k_n(2)}\oplus \cdots\oplus \M_{k_n(j_n)}
%\]
%where here $\M_j$ denotes the $j\times j$ matrices over $\CC$. Let $V_n = \{1,2, \dots, j_n\}$. It is a fact that if $A_n\hookrightarrow A_{n+1}$, then up to unitary equivalence we must have that each summand of $A_n$ embeds into the diagonal of some summand of $A_{n+1}$. To construct the edge set $E_n$, we draw an edge from $i\in V_n$ to $l\in V_{n+1}$ for each copy of $\M_{k_n(i)}$ sent inside $\M_{k_{n+1}(l)}$ for the embedding $A\hookrightarrow A_{n+1}$.
\subsection{The $3\times3$ matrices}
This example is a subexample of the previous example, but it will illustrate how we approach the following two examples.

Let $\I_3$ denote the symmetric inverse monoid on the three element set $\{1, 2, 3\}$. This is a Boolean inverse monoid which satisfies condition \eqref{condH}, and we define a map $\pi: \I_3\to \M_3$ by saying that 
\[
\pi(\phi)_{ij} = \begin{cases}1&\text{if }\phi(j) = i\\0&\text{otherwise.}
\end{cases}
\]
Then it is straightforward to verify that $\pi$ is in fact the universal Boolean inverse monoid representation of $\I_3$.

Now instead consider the subset $R_3\subset \I_3$ consisting of the identity, the empty function, and all functions with domain consisting of one element. Then $R_3$ is an inverse monoid, and $\pi(R_3)$ is the set of all matrix units together with the identity matrix and zero matrix. When restricted to $R_3$, $\pi$ is the universal tight representation of $R_3$. Hence $\Ct(R_3)\cong \CB(\I_3)\cong \M_3$. 

There is only one invariant mean $\mu$ on $\I_3$ -- for an idempotent Id$_U\in \I_3$, we have $\mu($Id$_U) = \frac13|U|$. The tight groupoid of $R_3$ is the equivalence relation $\{1, 2, 3\}\times \{1, 2, 3\}$, which is principal -- we also have that $\gt(R_3)^a \cong \I_3$. The unique invariant mean on $\I_3$ is identified with the unique normalized trace on $\M_3$.

Our last two examples follow this mold, where we have an inverse monoid $S$ which generates a C*-algebra $\Ct(S)$, and we relate the traces of $\Ct(S)$ to the invariant means of $\gt(S)^a$. 

\subsection{Self-similar groups}
Let $X$ be a finite set, let $G$ be a group, and let $X^*$ denote the set of all words in elements of $X$, including an empty word $\varnothing$. Let $X^\omega$ denote the Cantor set of one-sided infinite words in $X$, with the product topology of the discrete topology on $X$. For $\alpha\in X^*$, let $C(\alpha) = \{\alpha x\mid x\in X^\omega\}$ -- sets of this type are called {\em cylinder sets} and form a clopen basis for the topology on $X$.

Suppose that we have a faithful length-preserving action of $G$ on $X^*$, with $(g, \alpha)\mapsto g\cdot \alpha$, such that for all $g\in G$, $x\in X$ there exists a unique element of $G$, denoted $\left.g\right|_x$, such that for all $\alpha\in X^*$
\[
g(x\alpha) = (g\cdot x)(\left.g\right|_x\cdot \alpha).
\]
In this case, the pair $(G,X)$ is called a {\em self-similar group}. The map $G\times X\to G$, $(g, x)\mapsto \left.g\right|_x$ is called the {\em restriction} and extends to $G\times X^*$ via the formula
\[
\left.g\right|_{\alpha_1\cdots \alpha_n} = \left.g\right|_{\alpha_1}\left.\right|_{\alpha_2}\cdots \left.\right|_{\alpha_n}
\]
and this restriction has the property that for $\alpha, \beta\in X^*$, we have
\[
g(\alpha\beta) = (g\cdot \alpha)(\left.g\right|_{\alpha}\cdot \beta).
\]
The action of $G$ on $X^*$ extends to an action of $G$ on $X^\omega$ given by
\[
g\cdot (x_1x_2x_3\dots) = (g\cdot x_1)(\left.g\right|_{x_1}\cdot x_2)(\left.g\right|_{x_1x_2}\cdot x_3)\cdots
\]

In \cite{Nek09}, Nekrashevych associates a C*-algebra to $(G, X)$, denoted $\mathcal{O}_{G, X}$, which is the universal C*-algebra generated by a set of isometries $\{s_x\}_{x\in X}$ and a unitary representation $\{u_g\}_{g\in G}$ satisfying
\begin{enumerate}
\item[(i)]$s_x^*s_y = 0$ if $x\neq y$, 
\item[(ii)]$\sum_{x\in X}s_xs_x^* = 1$,
\item[(iii)]$u_gs_x = s_{g\cdot x}u_{\left.g\right|_x}$.
\end{enumerate}

One can also express $\mathcal{O}_{G,X}$ as the tight C*-algebra of an inverse semigroup. Let
\[
S_{G, X} = \{ (\alpha, g, \beta)\mid \alpha,\beta\in X^*, g\in G\}\cup\{0\}.
\]
This set becomes an inverse semigroup when given the operation
\[
(\alpha, g, \beta)(\gamma, h, \nu) = \begin{cases}(\alpha (g\cdot\gamma'), \left.g\right|_{\gamma'}h, \nu), &\text{if }\gamma = \beta\gamma',\\ (\alpha, g(\left.h^{-1}\right|_{\beta'})^{-1}, \nu (h^{-1}\cdot\beta')), & \text{if } \beta = \gamma\beta',\\ 0 &\text{otherwise}\end{cases}
\]
with
\[
(\alpha, g, \beta)^* = (\beta, g^{-1}, \alpha).
\]
Here, $E(S_{X, G}) = \{(\alpha, 1_G, \alpha)\mid \alpha\in X^*\},$ and the tight spectrum $\Et(S_{G, X})$ is homeomorphic $X^\omega$ by the identification
\[
x\in X^\omega \mapsto \{(\alpha, 1_G, \alpha)\in E(S_{G, X})\mid \alpha\text{ is a prefix of }x\} \in \Et(S_{G, X}).
\] 
If $\theta$ is the standard action of $S_{G, X}$ on $\Et(S_{G, X})$, then $D^\theta_{(\alpha, 1_G, \alpha)} = C(\alpha)$. If $s = (\alpha,g, \beta)\in S_{X, G}$, then
\[
\theta_s: C(\beta)\to C(\alpha)
\]
\[
\theta_s(\beta x) = \alpha(g\cdot x)
\]
It is shown in \cite{EP14a} that $\mathcal{O}_{G, X}$ is isomorphic to $\Ct(S_{G, X})$.

We show that the universal tight representation of $S_{G, X}$ is faithful. This will be accomplished if we can show that the map from $S_{G, X}$ to $\gt(S_{G, X})^a$ given by
\[
s \mapsto \Theta(s, D^\theta_{s^*s})
\]
is injective. If $s = (\alpha, g, \beta)$, then 
\[
\Theta(s, D^\theta_{s^*s}) = \{ [(\alpha, g, \beta), \beta x] \mid x\in X^\omega\}.
\]
It is straightforward that $d(\Theta(s, D^\theta_{s^*s})) = C(\beta)$ and $r(\Theta(s, D^\theta_{s^*s}))= C(\alpha)$. Suppose we have another element $t = (\gamma, h, \eta)$ such that $\Theta(s, D^\theta_{s^*s}) = \Theta(t, D^\theta_{t^*t})$. Since these two bisections are equal, their sources (resp. ranges) must be equal, so $C(\beta) = C(\eta)$ (resp. $C(\alpha) = C(\gamma)$). Hence, $\alpha = \gamma$ and $\beta = \eta$. Since $r$ and $d$ are both bijective on these slices, we must have that for all $\beta x \in C(\beta)$,  $\alpha (g\cdot x) = \alpha(h\cdot x)$. Hence for all $x\in X^\omega$, we must have that $g\cdot x = h\cdot x$. The action of $G$ on $X^*$ is faithful, so the induced action of $G$ on $X^\omega$ is also faithful, hence $g = h$ and so $t = s$. 

As it stands, the Boolean inverse monoid $\gt(S_{G,X})^a$ cannot have any invariant means.  This is because the subalgebra of $\mathcal{O}_{G, X}$ generated by $\{s_x\mid x\in X\}$ is isomorphic to the Cuntz algebra $\mathcal{O}_{|X|}$, and a trace on $\mathcal{O}_{G, X}$ would have to restrict to a trace on $\mathcal{O}_{|X|}$, which is purely infinite and hence has no traces. 

To justify the inclusion of this example in this paper about invariant means, we restrict to an inverse subsemigroup of $S_{G, X}$ whose corresponding ample semigroup will admit an invariant mean. Let
\[
S^{=}_{G, X} = \{ (\alpha, g, \beta)\in S_{G, X}\mid |\alpha| = |\beta|\}\cup\{0\}.
\]
One can easily verify that this is closed under product and involution, and so is a inverse subsemigroup of $S_{G, X}$, with the same set of idempotents as $S_{G, X}$. If $\alpha, \beta\in X^*$, $|\alpha| = |\beta|$, and $g\in G$, then 
\[
(\alpha, g, \beta)^*(\alpha, g, \beta) = (\beta, 1_G, \beta), \hspace{1cm}(\alpha, g, \beta)(\alpha, g, \beta)^* = (\alpha,1_G, \alpha).
\]
If $\mu$ were an invariant mean on $\gt(S^=_{G, X})^a$, then we would have to have, for all $\alpha, \beta\in X^*$ and $|\alpha| = |\beta|$, that $\mu(C(\alpha)) = \mu(C(\beta))$. Moreover, for a given length $n$, the set $\{C(\alpha)\mid |\alpha| = n\}$ forms a disjoint partition of $X^\omega$, and so we must have 
\begin{equation}\label{meanonSSG}
\mu(C(\alpha)) = |X|^{-|\alpha|}.
\end{equation}
Any clopen subset of $X^\omega$ must be a finite disjoint union of cylinders. Hence the map $\mu$ on $E(\gt(S^=_{G, X})^a)$ determined by \eqref{meanonSSG} is an invariant mean, and is in fact the unique invariant mean on $\gt(S^=_{G, X})^a$. 

In the general case, it is possible for $\gt(S^=_{G, X})$ to be neither Hausdorff nor principal. We now give an explicit example where we get a unique trace to go along with our unique invariant mean.

\begin{ex}{\bf (The 2-odometer)}

Let $X = \{0, 1\}$ and let $\ZZ = \left\langle z\right\rangle$ be the group of integers with identity $e$ written multiplicatively. The {\em 2-odometer} is the self-similar group $(\ZZ, X)$ determined by
\[
z\cdot 0 = 1\hspace{1cm} \left.z\right|_0 = e
\]
\[
z\cdot 1 = 0\hspace{1cm} \left.z\right|_1 = z.
\]
If one views a word $\alpha\in X^*$ as a binary number (written backwards), then $z\cdot \alpha$ is the same as 1 added to the binary number for $\alpha$, truncated to the length of $\alpha$ if needed. If such truncation is not needed, $\left.z\right|_\alpha = e$, but if truncation is needed, $\left.z\right|_\alpha = z$.

The action of $\ZZ$ on $\{0,1\}^\omega$ induced by the 2-odometer is the familiar Cantor minimal system of the same name. For $x\in \{0,1\}^\omega$ we have
\[
z\cdot x = \begin{cases}000\cdots & \text{if }x_i = 1\text{ for all }i\\
                        00\cdots01x_{i+1}x_{i+2}\cdots & \text{if }x_i = 0 \text{ and }x_j = 1 \text{ for all }j< i
\end{cases}
\]
This action of $\ZZ$ is {\em free} (i.e. $z^n\cdot x = x$ implies $n =0$) and {\em minimal} (i.e. the set $\{z^n\cdot x\mid n\in\ZZ\}$ is dense in $\{0,1\}^\omega$ for all $x\in \{0,1\}^\omega$). 

\begin{lem}
The groupoid of germs $\gt(S^=_{\ZZ, X})$ is principal.
\end{lem}
\begin{proof}
Take $x,y\in \{0,1\}^\omega$ and suppose that we have $\alpha, \beta\in \{0,1\}^*$ with $|\alpha| = |\beta|$ and $n\in \ZZ$ such that $[(\alpha, z^n, \beta), x]\in \gt(S^=_{\ZZ, X})$ and $r([(\alpha, z^n, \beta), x]) = y$. This implies that $x = \beta v$ for some $v\in \{0,1\}^\omega$, and that $y = \alpha (z^n\cdot v)$. Suppose we can find another germ from $x$ to $y$, that is, suppose we have $\gamma, \eta\in \{0,1\}^*$ with $|\gamma| = |\eta|$ and $m\in \ZZ$ such that $[(\gamma, z^n, \eta), x]\in \gt(S_{\ZZ, X})$ and $r([(\gamma, z^n, \eta), x]) = y$. Again we can conclude that $x = \eta u$ for some $u\in \{0,1\}^\omega$, and that $y = \gamma (z^m\cdot u)$. There are two cases.

Suppose first that $\beta = \eta\delta$ for some $\delta\in \{0,1\}^*$.   Then $\eta\delta v = x = \eta u$, and so $\delta v = u$. We also have $\alpha (z^n\cdot v) = y = \gamma (z^m\cdot u) $. Because $|\alpha|=|\beta| \geq |\eta| = |\gamma|$, this implies that there exists $\nu\in \{0,1\}^*$ with $|\nu| = |\delta|$ and $\alpha = \gamma\nu$. Hence $\nu (z^n\cdot v) = z^m\cdot u = (z^m\cdot \delta)\left.z^m\right|_\delta \cdot v$, which gives us that $\nu = z^m\cdot \delta$ and $z^n\cdot v = \left.z^m\right|_\delta\cdot v$, and since the action on $\{0,1\}^\omega$ is free we have $z^n = \left.z^m\right|_\delta$. 

So we have that $x\in C(\beta) = D^\theta_{(\beta, e, \beta)}$, and we calculate
\begin{eqnarray*}
(\gamma, z^n, \eta)(\beta, e, \beta) &=& (\gamma(z^m\cdot \delta), \left.z^m\right|_\delta, \beta) = (\gamma\nu, z^n, \beta) = (\alpha, z^n, \beta) \\
&=& (\alpha, z^n, \beta)(\beta, e, \beta)
\end{eqnarray*}
where the first equality is by the definition of the product. Hence $[(\alpha, z^n, \beta), x] = [(\gamma, z^n, \eta), x]$. The case where $\beta$ is shorter than $\eta$ is similar. Hence, $\gt(S^=_{\ZZ, X})$ is principal.
\end{proof}
It is routine to check that $S^=_{\ZZ, X}$ satisfies condition \eqref{condH} (in fact, it is E*-unitary, see \cite[Example 3.4]{ES16}). The groupoid $\gt(S^=_{\ZZ, X})$ is amenable, see \cite[Proposition 5.1.1]{AR00} and \cite[Corollary 10.18]{EP13}. Hence Theorem \ref{maintheo} applies, and there is only one normalized trace on $\Ct(S^=_{\ZZ, X})$, the one arising from the invariant mean.

As the observant reader is no doubt aware at this point, $\Ct(S^=_{\ZZ, X})$ is nothing more than the crossed product $C(\{0,1\}^\omega)\rtimes \ZZ$ arising from the usual odometer action \cite[Theorem 7.2]{Nek04}, which has a unique normalized trace due to the dynamical system $(\{0,1\}^\omega, \ZZ)$ having a unique invariant measure (given by \eqref{meanonSSG}).
\end{ex}

\subsection{Aperiodic tilings}

We close with another example where the traces on the relevant C*-algebras are known beforehand, and hence give us invariant means. 

A {\em tile} is a closed subset of $\RR^d$ homeomorphic to the closed unit ball. A {\em partial tiling} is a collection of tiles in $\RR^d$ with pairwise disjoint interiors, and the {\em support} of a partial tiling is the union of its tiles. A {\em patch} is a finite partial tiling, and a {\em tiling} is a partial tiling with support equal to $\RR^d$. If $P$ is a partial tiling and $U\subset \RR^d$, then let $P(U)$ be the partial tiling of all tiles in $P$ which intersect $U$. A tiling $T$ is called {\em aperiodic} if $T+x \neq T$ for all $0\neq x\in \RR^d$.

Let $T$ be a tiling. We form an inverse semigroup $S_T$ from $T$ as follows. For a patch $P\subset T$ and tiles $t_1, t_2\in P$ we call the triple $(t_1, P, t_2)$ a {\em doubly pointed patch}. We put an equivalence relation on such triples, by saying that $(t_1, P, t_2)\sim (r_1, Q, r_2)$ if there exists a vector $x\in \RR^d$ such that $(t_1+x, P+x, t_2+x)= (r_1, Q, r_2)$, and let $[t_1, P, t_2]$ denote the equivalence class of such a triple -- this is referred to a {\em doubly pointed patch class}. Let
\[
S_T = \{[t_1, P, t_2] \mid (t_1, P, t_2)\text{ is doubly pointed patch }\}\cup \{0\}
\]
be the set of all doubly pointed patch classes together with a zero element. If $[t_1, P, t_2],[r_1, Q, r_2]$ are two elements of $S_T$, we let
\[
[t_1, P, t_2][r_1, Q, r_2] = \begin{cases}
[t_1, P\cup Q', r_2']& \text{if there exists }(r_1', Q', r_2')\in [r_1, Q, r_2]\\
 &  \text{such that }r_1' = t_2 \text{ and }P\cup Q' \text{ is }\\
 & \text{a patch in }T+x\text{ for some }x\in \RR^d\\
0 & \text{otherwise},
\end{cases}
\]
and define all products involving 0 to be 0. Also, let $[t_1, P, t_2]^* = [t_2, P, t_1]$. With these operations, $S_T$ is an inverse semigroup. This inverse semigroup was defined by Kellendonk \cite{KelCoin} \cite{Kel97}, and is E*-unitary.

Suppose there exists a finite set $\p$ of tiles each of which contain the origin in the interior such that for all $t\in T$, there exists $x_t\in\RR^d$ and $p\in \p$ such that $t = p+ x_t$. In this case, $\p$ is called a set of {\em prototiles} for $T$. By possibly adding labels, we may assume that $x_t$ and $p$ are unique -- we call $x_t$ the {\em puncture} of $t$. Consider the set
\[
X_T = \{T-x_t\mid t\in T\}
\]
and put a metric on $X_T$ by setting
\[
d(T_1, T_2) = \inf\{1, \epsilon\mid T_1(B_{1/\epsilon}(0)) = T_1(B_{1/\epsilon}(0))\}
\]
and let $\Op$ denote the completion of $X_T$ in this metric (above, $B_r(x)$ denotes the open ball in $\RR^d$ of radius $r$ around $x\in \RR$). One can show that all elements of $\Op$ are tilings consisting of translates of $\p$ which also contain an element of $\p$ and that the metric above extends to the same metric on $\Op$ -- this is called the {\em punctured hull} of $T$. 

We make the following assumptions on $T$:

\begin{enumerate}
\item $T$ has {\em finite local complexity} if for any $r>0$, there are only finitely many patches in $T$ with supports having outer radius less than $r$, up to translational equivalence. 
\item $T$ is {\em repetitive} if for every patch $P\subset T$, there exists $R>0$ such that every ball of radius $R$ in $\RR^d$ contains a translate of $P$.
\item $T$ is {\em strongly aperiodic} if all elements of $\Op$ are aperiodic.
\end{enumerate}
In this case $\Op$ is homeomorphic to the Cantor set. For a patch $P\subset T$ and tile $t\in P$, let
\[
U(P, t) = \{T'\in \Op\mid P - x_t\subset T'\}.
\]
Then these sets are clopen in $\Op$ and generate the topology. Let
\[
\Rp = \{(T_1, T_1 + x)\in \Op\times\Op\mid x\in \RR^d\}
\]
and view this equivalence relation as a principal groupoid. Endow it with the topology inherited by viewing it as a subspace of $\Op\times \RR^d$. For a patch $P\subset T$ and $t_1, t_2\in P$, let
\[
V(t_1, P, t_2) = \{(T_1, T_2)\in \Rp\mid T_1\in U(P, t_1), T_2 = T_1 + x_{t_1}- x_{t_2}\}
\]
Then these sets are compact bisections in $\Rp$, and generate the topology on $\Rp$. This groupoid is Hausdorff, ample, and amenable \cite{PS99}. The C*-algebra of $\Rp$ was defined by Kellendonk in \cite{KelNCG} (denoted there $A_T$) and studied further in \cite{KP00}, \cite{Pu00}, \cite{Put10}, \cite{Ph05}, \cite{St12}.

\begin{figure}
\includegraphics[scale=0.52]{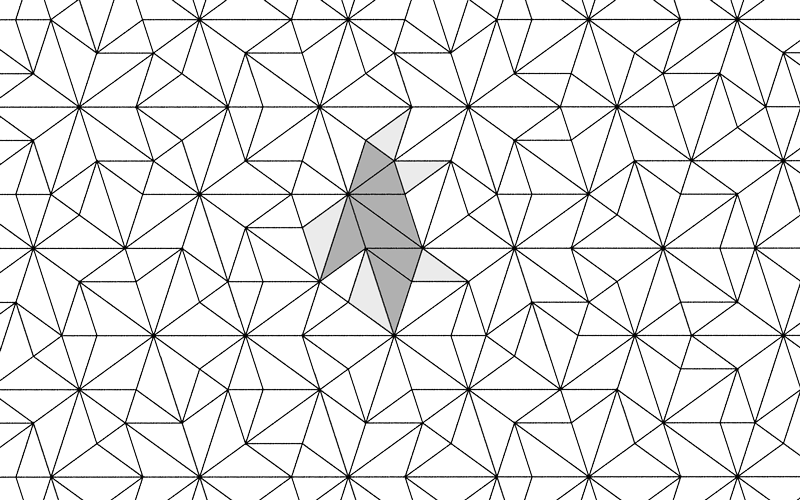}
\caption{In the Robinson triangles version of the Penrose tiling, each triangle is always next to a similar triangle with which it forms a rhombus. Let $P$ be the dark gray patch, and let $P'$ be the patch with the lighter gray tiles added. Then for any dark gray tile $t$, $U(P, t) = U(P', t)$}
\centering
\end{figure}

We proved in \cite[Theorem 3]{EGS12} that $\gt(S_T)\cong \Rp$ -- the universal tight representation of $S_T$ maps $[t_1, P, t_2]$ to the characteristic function of $V(t_1, P, t_2)$.\footnote{The same result follows from \cite[Section 9]{L08} combined with the fact that Lenz's groupoid coincides with the tight groupoid when $\Et(S) = \Eu(S)$, see \cite[Theorem 5.15]{LL13}.} It is interesting to note that in this case that the universal tight representation may not be faithful. Suppose that we could find $P\subset P'$, both patches in $T$, and that $P+x\subset T$ can only happen if $P'+x\subset T$. Then for a tile $t\in P$, the two idempotents $[t, P, t], [t, P', t]$ are different elements in $S_T$, but are both mapped to the characteristic function of $U(P,t) = U(P',t)$ under the universal tight representation -- indeed, $[t, P', t]$ is dense in $[t, P, t]$, see Figure 1. We note that \cite{L08}, \cite{Ex09}, and \cite{LL13} address other cases where the tight representation may not be faithful.

The C*-algebra $A_T$ can be seen as the C*-algebra of a Boolean inverse monoid, namely $\gt(S_T)^a$ -- one could then rightly call this the {\em Boolean inverse monoid associated to $T$}. The traces of $A_T$ are already well-studied, see \cite{KP00}, \cite{Pu00}. Often, as is the case with the Penrose tiling, there is a unique trace, see \cite{Pu00}. 

\begin{theo}
Let $T$ be a tiling which satisfies conditions 1--3 above, and let $\gt(S_T)^a$ be the Boolean inverse monoid associated to $T$. Then $M(\gt(S_T)^a)\cong T(A_T)\cong IM(\Rp)$.
\end{theo}

{\bf Acknowledgment}: I thank Thierry Giordano for an enlightening conversation about this work, and I also thank Ganna Kudryavtseva, Ruy Exel, and Daniel Lenz for comments and suggestions on a preprint version of this paper. Finally, I thank the referee for a thorough reading.

\bibliographystyle{alpha}
\bibliography{C:/Users/Charles/Dropbox/Research/bibtex}{}

{\small 
\textsc{University of Ottawa, Department of Mathematics and Statistics. 585 King Edward, Ottawa, ON, Canada, K1N 6N5} \texttt{cstar050@uottawa.ca} 
}
\end{document}